\documentclass[oneside,english]{amsart}
\usepackage{graphicx}
\usepackage[T1]{fontenc}
\usepackage[latin9]{inputenc}
\usepackage{geometry}
\usepackage{amstext}
\usepackage{amsthm}
\usepackage{amssymb}
\usepackage{setspace}

\makeatletter
\numberwithin{equation}{section}
\numberwithin{figure}{section}
  \theoremstyle{plain}
  \newtheorem*{thm*}{\protect\theoremname}
  \theoremstyle{plain}
  \newtheorem*{cor*}{\protect\corollaryname}
 \theoremstyle{definition}
 \newtheorem*{defn*}{\protect\definitionname}
\theoremstyle{plain}
\newtheorem{thm}{\protect\theoremname}
  \theoremstyle{definition}
  \newtheorem{example}[thm]{\protect\examplename}
  \theoremstyle{remark}
  \newtheorem*{rem*}{\protect\remarkname}
  \theoremstyle{remark}
  \newtheorem{rem}[thm]{\protect\remarkname}
  \theoremstyle{plain}
  \newtheorem{cor}[thm]{\protect\corollaryname}

\usepackage{babel}

\makeatother

\usepackage{babel}
  \providecommand{\corollaryname}{Corollary}
  \providecommand{\definitionname}{Definition}
  \providecommand{\examplename}{Example}
  \providecommand{\remarkname}{Remark}
  \providecommand{\theoremname}{Theorem}
\providecommand{\theoremname}{Theorem}

\begin{document}

\title[Toeplitz operators with non-harmonic symbol]{Hyponormal Toeplitz operators with non-harmonic Symbol acting on the Bergman space}

\author{Matthew Fleeman}
\address{M.~Fleeman: Department of Mathematics, Baylor University, One Bear Place \#97328,      
 Waco, TX  76798, USA}
\email{Matthew$\underline{\,\,\,}$Fleeman@baylor.edu}

\author{Constanze Liaw}
\address{C.~Liaw: Department of Mathematical Sciences, University of Delaware, 501 Ewing Hall, Newark, DE 19716, USA. And CASPER, Baylor University, One Bear Place \#97328,      
 Waco, TX  76798, USA}
\email{Liaw@udel.edu}

\thanks{
The work of C.~Liaw was supported by Simons Foundation Grant \#426258.}


\begin{abstract}
The Toeplitz operator acting on the Bergman space $A^{2}(\mathbb{D})$,
with symbol $\varphi$ is given by $T_{\varphi}f=P(\varphi f)$, where
$P$ is the projection from $L^{2}(\mathbb{D})$ onto the Bergman
space. 
We present some history on the study of
hyponormal Toeplitz operators acting on $A^{2}(\mathbb{D})$, as well
as give results for when $\varphi$ is a non-harmonic polynomial. We include a first investigation of Putnam's inequality for hyponormal operators with non-analytic symbols.
Particular attention is given to unusual hyponormality behavior that arises due to the extension of the class of allowed symbols. 
\end{abstract}

\maketitle

\section{Introduction}

Let $H$ be a complex Hilbert space and $T$ be a bounded linear operator
acting on $H$ with adjoint $T^{*}$. Operator $T$ is said to be
\emph{hyponormal} if $[T^{*},T]:=T^{*}T-TT^{*}\geq0$. That is, if
for all $u\in H$ 
\[
\left\langle [T^{*},T]u,u\right\rangle \geq0.
\]

The study of hyponormal operators is strongly related to the spectral
and perturbation theories of Hilbert space operators, singular integral
equations, and scattering theory. The interested reader is referred
to the monograph \cite{MartinPutinar} by M.~Martin and M.~Putinar.
One particularly interesting result for hyponormal operators, Putnam's
inequality, states that if $T$ is hyponormal, then 
\[
\Vert[T^{*},T]\Vert\leq\frac{\mathrm{Area}(\sigma(T))}{\pi},
\]
where $\sigma(T)$ denotes the spectrum of $T$ (cf.~\cite{AxlerShapiro}).

We study the hyponormality of certain operators acting on the Bergman
space 
\[
A^{2}(\mathbb{D})=\left\{ f\in\mathrm{Hol}(\mathbb{D}):\int_{\mathbb{D}}\left|f(z)\right|^{2}dA(z)<\infty\right\} .
\]
Let $\varphi\in L^{\infty}(\mathbb{D}).$ The \emph{Toeplitz operator}
$T_{\varphi}$ is given by 
\[
T_{\varphi}f=P(\varphi f)\qquad f\in A^{2}(\mathbb{D}),
\]
where $P$ is the orthogonal projection from $L^{2}(\mathbb{D})$
onto $A^{2}(\mathbb{D}).$

In the Hardy space setting the question of when $T_{\varphi}$ is
hyponormal for $\varphi\in L^{\infty}(\mathbb{T})$ was answered by
C.~Cowen in \cite{Cowen}, who proved the following theorem:
\begin{thm*}
Let $\varphi\in L^{\infty}(\mathbb{T})$ be given by $\varphi=f+\bar{g}$,
with $f,g\in H^{2}$. Then $T_{\varphi}$ is hyponormal if and only
if 
\[
g=c+T_{\bar{h}}f,
\]
for some constant $c$ and some $h\in H^{\infty}(\mathbb{D})$, with
$\left\Vert h\right\Vert _{\infty}\leq1$.
\end{thm*}
This completely characterized hyponormal Toeplitz operators acting
on the Hardy space. Cowen's proof relies on a dilation theorem of D.~Sarason
\cite[Theorem 1]{Sarason}, and the fact that $\left(H^{2}\right)^{\perp}$
is just the conjugates of $H^{2}$ functions which vanish at the origin. 

In the Bergman space setting, where we lack an analog to Sarason's
dilation theorem, and where $\left(A^{2}\right)^{\perp}$ is a much
larger space, a similar characterization is lacking. One of the principle
difficulties in exploring questions of hyponormality originates from
the behavior of the self-commutator under operator addition. In particular,
if we let $u$ be in a complex Hilbert space $H$, and $T$ and $S$
be operators on $H$, then we find
\begin{align}
 & \quad\,\left\langle \left[(T+S)^{*},T+S\right]u,u\right\rangle \nonumber \\
 & =\left\langle Tu,Tu\right\rangle -\left\langle T^{*}u,T^{*}u\right\rangle +2\mathrm{Re}\left[\left\langle Tu,Su\right\rangle -\left\langle T^{*}u,S^{*}u\right\rangle \right]+\left\langle Su,Su\right\rangle -\left\langle S^{*}u,S^{*}u\right\rangle .\label{eq:crossterms}
\end{align}
As we shall see, the ``cross-terms'' $2\mathrm{Re}\left[\left\langle Tu,Su\right\rangle -\left\langle T^{*}u,S^{*}u\right\rangle \right]$
lead to many somewhat unexpected results which reveals a subtlety
in the study of hyponormal operators. The explicit expressions in
\eqref{eq:crossterms} lead to involved series computations. Our primary
effort consists of extracting reasonable necessary and/or sufficient
conditions from series corresponding to several different types of
non-harmonic symbols. It is worth noting that if both $T$ and $S$
are Toeplitz operators with harmonic symbols, then these cross terms
vanish, which leads to a smoother study of such operators, e.g. in
\cite{AhernCuckovic}, \cite{Hwang}, and \cite{Sadraoui}.

One of the central questions this paper explores is the following: 
\begin{center}
{\em Given a hyponormal Toeplitz operator $T_{\varphi}$ acting
on $A^{2}(\mathbb{D})$ and a symbol $\psi\in L^{\infty}(\mathbb{D})$,\\
 when is $T_{\varphi+\psi}$ hyponormal?} 
\par\end{center}

When $\psi$ is not harmonic, this question turns out to be particularly
elusive. As we shall see in Section \ref{sec:Non-Harmonic symbols},
even requiring that $T_{\psi}$ be self-adjoint is not enough to guarantee
the hyponormality of $T_{\varphi+\psi}$.

We are also interested in some spectral properties of hyponormal
$T_{\varphi}$, especially because the commutator has interesting interactions with the geometry of the image $\varphi(\mathbb{D})$.  It is an immediate consequence of Putnam's inequality and the spectral mapping theorem (cf.~\cite[p.~263]{Rudin}) that the norm of the commutator of $T_{\varphi}^{*}$ and  $T_{\varphi}$ is bounded above by $\mathrm{Area}(\varphi(\mathbb{D}))/\pi$ for \emph{analytic} $\varphi$, and in \cite{OlsenReguera} it was shown that this bound can be improved to $\mathrm{Area}(\varphi(\mathbb{D}))/(2\pi)$ for analytic and univalent $\varphi$. In \cite{FleemanKhavinson}, it was conjectured that the hypothesis ``univalent" is superfluous for this stronger bound. We extend this conjecture to non-analytic symbols.

The paper proceeds as follows: In Section \ref{sec:Harmonic Symbols},
we give an overview of some known results for the hyponormality results
of Toeplitz operators with harmonic symbols. This overview is by
no means exhaustive, but gives a flavor for the types of results in
this area to date. Of particular note is that questions of hyponormality
even of operators with harmonic polynomials as symbols have still
not been completely answered, as well as the elusiveness of both necessary
and sufficient conditions for hyponormality. In Section \ref{sec:Non-Harmonic symbols},
we focus on operators with symbols which are not harmonic. We give
several sufficient conditions for the hyponormality of certain operators
whose symbol is a non-harmonic polynomial, as well as several examples
which indicate that the situation is rather subtle. In Section
\ref{sec:Polynomials of Fixed Relative Degree}, we look at operators
whose symbols satisfy $\varphi(z)=a_{1}z^{m_{1}}\bar{z}^{n_{1}}+\ldots+a_{k}z^{m_{k}}\bar{z}^{n_{k}}$,
with $m_{1}-n_{1}=\ldots=m_{k}-n_{k}=\delta\geq0$.
Finally, in Section \ref{s-spectral}, we show that the norm of the commutator of $T_{\varphi}^{*}$ and  $T_{\varphi}$ is bounded by 1/2 for $\varphi(z)=z^{m}\bar{z}^{n}$ with $m>n$.

\noindent \textbf{Acknowledgement.} Many thanks to D.~Khavinson for
inspiring discussions, and to C.~Cowen for his very helpful correspondance
and encouragement.


\section{Toeplitz operators with harmonic symbol\label{sec:Harmonic Symbols}}

The study of hyponormal operators with harmonic symbols is greatly
simplified by the lack of cross-terms. In particular, if $\varphi=f+\bar{g}$
where $f$ and $g$ are holomorphic and bounded in $\mathbb{D}$ then
one may show that the cross-term $2\mathrm{Re}\left[\left\langle T_{f}u,T_{\bar{g}}u\right\rangle -\left\langle T_{\bar{f}}u,T_{g}u\right\rangle \right]$
vanishes. Thus, one can show the hyponormality of $T_{\varphi}$ by
showing that $\left\Vert H_{\bar{f}}u\right\Vert ^{2}\geq\left\Vert H_{\bar{g}}u\right\Vert ^{2}$
for all $u$ in the Bergman space, where $H_{\bar{\varphi}}$ is the
\emph{Hankel operator} $I-T_{\bar{\varphi}}$.

In \cite{Sadraoui}, H.~Sadraoui examined the hyponormality of Toeplitz
operators $T_{\varphi}$ acting on the Bergman space when $\varphi$
is harmonic. One of his first results, \cite[Prop.~1.4.3]{Sadraoui},
gave a necessary boundary condition for $f$ and $g$ whenever $f'$
is in the Hardy space. This result is particularly interesting because
in the Bergman space, boundary value results are so rare. 
\begin{thm*}
Let $f$ and $g$ be bounded analytic functions, such that $f'\in H^{2}$.
If $T_{f+\bar{g}}$ is hyponormal, then $g'\in H^{2}$ and $\left|g'\right|\leq\left|f'\right|$
almost everywhere on $\mathbb{T}$. 
\end{thm*}
He also showed that this result is sharp, but not in general sufficient.
In particular, he proved the following theorem \cite[Prop.~1.4.4]{Sadraoui}
for harmonic polynomials. 
\begin{thm*}
Consider the operator $T_{z^{n}+\alpha\bar{z}^{m}}$.

1. If $m\leq n$, then $T_{z^{n}+\alpha\bar{z}^{m}}$ is hyponormal
if and only if $\left|\alpha\right|\leq\sqrt{\frac{m+1}{n+1}}$.

2. If $m\geq n,$ $T_{z^{n}+\alpha\bar{z}^{m}}$ is hyponormal if
and only if $\left|\alpha\right|\leq\frac{n}{m}$. 
\end{thm*}
This leads to a host of examples where $\left|g'\right|\leq\left|f'\right|$
on $\mathbb{T},$ but $T_{f+\bar{g}}$ is not hyponormal. In \cite[Theorem 4]{AhernCuckovic},
P.~Ahern and Z.~\v{C}u\v{c}kovi\'{c} showed the following result giving
another necessary, but not sufficient, condition for the hyponormality
of $T_{\varphi}$ when $\varphi$ is harmonic. 
\begin{thm*}
Suppose $f$ and $g$ are holomorphic in $\mathbb{D}$ and $\varphi=f+\bar{g}\in L^{\infty}(\mathbb{D})$.
If $T_{\varphi}$ is hyponormal then $Tu\geq u$ in $\mathbb{D}$
where $u=\left|f\right|^{2}-\left|g\right|^{2}$.
\end{thm*}
Using this, they were able to show, as a corollary, a more general
version of Sadraoui's result. 
\begin{cor*}
Suppose $f$ and $g$ are holomorphic in $\mathbb{D}$, that $\varphi=f+\bar{g}$
is bounded in $\mathbb{D}$, and that $T_{\varphi}$ is hyponormal.
Then $\overline{\lim}_{z\rightarrow\zeta}\left(\left|f'(z)\right|^{2}-\left|g'(z)\right|^{2}\right)\geq0$
for all $\zeta\in\mathbb{T}$. In particular, if $f'$ and $g'$ are
continuous at $\zeta\in\mathbb{T}$, then $\left|f'(\zeta)\right|\geq\left|g'(\zeta)\right|$. 
\end{cor*}
Finally, in \cite{Hwang}, I.S.~Hwang proved the following theorem as part
of his study of hyponormal operators whose symbol is a harmonic polynomial.
We note here that the condition deals only with the modulus of the
coefficients of the given harmonic polynomial. 
\begin{thm*}
Let $f(z)=a_{m}z^{m}+a_{n}z^{n}$ and $g(z)=a_{-m}z^{m}+a_{-n}z^{n}$,
with $0<m<n$. If $T_{f+\bar{g}}$ is hyponormal and $\left|a_{n}\right|\leq\left|a_{-n}\right|$,
then we have 
\[
n^{2}\left(\left|a_{-n}\right|^{2}-\left|a_{n}\right|^{2}\right)\leq m^{2}\left(\left|a_{m}\right|^{2}-\left|a_{-m}\right|^{2}\right).
\]
\end{thm*}
Work continues to this day on the study of hyponormal Toeplitz operators
whose symbol is a harmonic polynomial. It is a testament to the subtlety
of the topic that even in this case there is still much to be said
about such symbols. Recently, in \cite{CuckovicCurto}, Z.~\v{C}u\v{c}kovi\'{c}
and R.~Curto proved the following result.
\begin{thm*}
Suppose $T_{\varphi}$ is hyponormal on $A^{2}(\mathbb{D})$ with
$\varphi(z)=\alpha z^{m}+\beta z^{n}+\gamma\bar{z}^{p}+\delta\bar{z}^{q}$,
where $m<n$ and $p<q$, and $\alpha,\beta,\gamma,\delta\in\mathbb{C}$.
Assume also that $n-m=q-p$. Then 
\[
\left|\alpha\right|^{2}n^{2}+\left|\beta\right|^{2}m^{2}-\left|\gamma\right|^{2}p^{2}-\left|\delta\right|^{2}q^{2}\ge2\left|\bar{\alpha}\beta mn-\bar{\gamma}\delta pq\right|.
\]
\end{thm*}
Note that in the above Theorems, only the moduli of the coefficients
are taken into account. As we shall see in Section \ref{sec:Polynomials of Fixed Relative Degree},
this is not necessarily the case when $\varphi$ is not harmonic.
We now turn our attention to such operators.


\section{Toeplitz operators with non-harmonic symbol\label{sec:Non-Harmonic symbols}}

So far, all of these results deal with Toeplitz operators whose symbol
is harmonic. The study of operators whose symbol is not harmonic
turns out to be more complicated because the cross-terms in equation \eqref{eq:crossterms}
do not vanish. 


\subsection{Simple non-harmonic symbols}

We begin our own investigations by looking at some simple examples.
We did not have to look far for some results which we found surprising.

It seemed heuristically plausible that adding a symbol corresponding
to a hyponormal Toeplitz operator to a symbol corresponding to a self-adjoint
Toeplitz operator should generate a hyponormal Toeplitz operator.
But this is not the case. 
\begin{example}
\label{ex:Hyponormal+Self-adjoint}
Operator $T_{z+C\left|z\right|^{2}}$ is not hyponormal when $C<-2\sqrt{2}$. 
\end{example}
\begin{proof}
We verify the statement in Example \ref{ex:Hyponormal+Self-adjoint}.
Let $\varphi_{n}(z)=\sqrt{\frac{n+1}{\pi}}z^{n}$. The collection
$\left\{ \varphi_{n}\right\} _{n=0}^{\infty}$ is the standard orthonormal
basis of $A^{2}(\mathbb{D}).$ Given $u(z)=\sum_{n=0}^{\infty}u_{n}\varphi_{n}\in A^{2}(\mathbb{D}),$
where $\left\{ u_{n}\right\} \in\ell^{2}$ we have that 
\[
T_{z}u=\sum_{n=0}^{\infty}\sqrt{\frac{n+1}{n+2}}u_{n}\varphi_{n+1},\quad\quad\text{and}\quad\quad
T_{\lvert z\rvert^{2}}u=\sum_{n=0}^{\infty}\frac{n+1}{n+2}u_{n}\varphi_{n}.
\]

Thus, we have that the cross-terms are 
\begin{align*}
2\mathrm{Re}\left[\left\langle T_{\lvert z\rvert^{2}}T_{z}u,u\right\rangle -\left\langle T_{z}T_{\lvert z\rvert^{2}}u,u\right\rangle \right] & =2\mathrm{Re}\left[\left\langle T_{z}u,T_{\lvert z\rvert^{2}}u\right\rangle -\left\langle T_{\bar{z}}u,T_{\lvert z\rvert^{2}}u\right\rangle \right]\\
 & =2\mathrm{Re}\left[\left\langle \sum_{n=0}^{\infty}\sqrt{\frac{n+1}{n+2}}\left(\frac{n+2}{n+3}-\frac{n+1}{n+2}\right)u_{n}\varphi_{n+1},\sum_{n=0}^{\infty}u_{n}\varphi_{n}\right\rangle \right]\\
 & =2\mathrm{Re}\sum_{n=0}^{\infty}\sqrt{\frac{n+1}{n+2}}\left(\frac{n+2}{n+3}-\frac{n+1}{n+2}\right)u_{n}\overline{u_{n+1}}.
\end{align*}

Now, by \cite{FleemanKhavinson} and \cite{OlsenReguera} we have
\[
\left\langle T_{z}u,T_{z}u\right\rangle -\left\langle T_{\bar{z}}u,T_{\bar{z}}u\right\rangle \leq\frac{1}{2}\left\Vert u\right\Vert ^{2},
\]
and since $T_{\lvert z\rvert^{2}}$ is self adjoint we have 
\[
\left\langle T_{\lvert z\rvert^{2}}u,T_{\lvert z\rvert^{2}}u\right\rangle -\left\langle T_{\lvert z\rvert^{2}}u,T_{\lvert z\rvert^{2}}u\right\rangle =0.
\]

If we then replace $T_{\lvert z\rvert^{2}}$ with $T_{C\lvert z\rvert^{2}}$, with real $C$, we
have the cross-terms 
\[
2\mathrm{Re}\left[\left\langle T_{z}u,T_{C\lvert z\rvert^{2}}u\right\rangle -\left\langle T_{\bar{z}}u,T_{C\lvert z\rvert^{2}}u\right\rangle \right]=2C\mathrm{Re}\sum_{n=0}^{\infty}\sqrt{\frac{n+1}{n+2}}\left(\frac{n+2}{n+3}-\frac{n+1}{n+2}\right)u_{n}\overline{u_{n+1}}.
\]

Thus we may choose $u\in A^{2}(\mathbb{D})$ and $C\in\mathbb{R}$,
such that 
\[
\frac{1}{2}\left\Vert u\right\Vert ^{2}+2C\mathrm{Re}\sum_{n=0}^{\infty}\sqrt{\frac{n+1}{n+2}}\left(\frac{n+2}{n+3}-\frac{n+1}{n+2}\right)u_{n}\overline{u_{n+1}}<0.
\]
For such a choice of $C$ then, operator $T_{z+C\lvert z\rvert^{2}}$ would
not be hyponormal. In particular if we choose $u(z)=\frac{1}{2}\varphi_{0}+\frac{1}{2}\varphi_{1}$,
then 
\[
\left\langle \left[T_{z+C\left|z\right|^{2}}^{*},T_{z+C\left|z\right|^{2}}\right]u,u\right\rangle =\frac{1}{6}+\frac{C}{12\sqrt{2}},
\]
which will be negative whenever $C<-2\sqrt{2}$. Thus, whenever we
have $C<-2\sqrt{2}$, we have that $T_{z+C\lvert z\rvert^{2}}$ is
not hyponormal. 
\end{proof}
At this point it is not known whether $-2\sqrt{2}$ is sharp. Because
of the form of the cross-terms, a test function of the form $u(z)=u_{0}\varphi_{0}(z)+u_{1}\varphi_{1}(z)$
of a given norm will have the largest possible contribution to the
final value of the self-commutator, however such a function function
also will have a relatively large value for $\left\langle T_{z}u,T_{z}u\right\rangle -\left\langle T_{\bar{z}}u,T_{\bar{z}}u\right\rangle $,
since $\left\langle T_{z}\varphi_{n},T_{z}\varphi_{n}\right\rangle -\left\langle T_{\bar{z}}\varphi_{n},T_{\bar{z}}\varphi_{n}\right\rangle \rightarrow0$
as $n\rightarrow\infty$. In particular $\left\langle T_{z}u,T_{z}u\right\rangle -\left\langle T_{\bar{z}}u,T_{\bar{z}}u\right\rangle =\frac{1}{2}\left\Vert u\right\Vert ^{2}$
only for $u=\varphi_{0}.$ Yet the example came as a surprise to us.
We had conjectured that the sum of a self-adjoint plus a hyponormal
symbol would always correspond to a hyponormal operator, and the above
simple counterexample was striking. 
\begin{thm}
\label{thm:Hyponormal monomials}Let $\varphi(z)=a_{m,n}z^{m}\bar{z}^{n}$,
with $m\geq n$ and $a_{m,n}\in\mathbb{C}$. Then $T_{\varphi}$ is
hyponormal. 
\end{thm}

\begin{proof}
It is a well known fact (cf.~\cite[Chapter 2, Lemma 6]{DurenSchuster})
that 
\[
P(z^{m}\bar{z}^{n})=\begin{cases}
\frac{m-n+1}{m+1}z^{m-n} \quad & m\geq n\\
0 & m<n.
\end{cases}
\]
Thus, if we let $u(z)=\sum_{k=0}^{\infty}u_{k}z^{k}\in A^{2}(\mathbb{D}),$
then we have
\[
P(z^{m}\bar{z}^{n}u)=\begin{cases}
\sum_{k=0}^{\infty}\frac{m+k-n+1}{m+k+1}u_{k}z^{m+k-n} \quad& m\geq n\\
\sum_{k=n-m}^{\infty}\frac{m+k-n+1}{m+k+1}u_{k}z^{m+k-n} \quad& m<n.
\end{cases}
\]
Taking into account that $T_{\varphi}^{*}=T_{\bar{\varphi}}$, we
find that 
\begin{align}
 & \,\quad\left\langle [T_{\varphi}^{*},T_{\varphi}]u,u\right\rangle \nonumber \\
 & =\left\langle T_{\varphi}u,T_{\varphi}u\right\rangle -\left\langle T_{\varphi}^{*}u,T_{\varphi}^{*}u\right\rangle \nonumber \\
 & =\left|a_{m,n}\right|^{2}\left(\sum_{k=0}^{\infty}\frac{m+k-n+1}{(m+k+1)^{2}}\left|u_{k}\right|^{2}-\sum_{k=m-n}^{\infty}\frac{n+k-m+1}{(n+k+1)^{2}}\left|u_{k}\right|^{2}\right)\nonumber \\
 & =\left|a_{m,n}\right|^{2}\left(\sum_{k=0}^{m-n-1}\frac{m+k-n+1}{(m+k+1)^{2}}\left|u_{k}\right|^{2}+\sum_{k=m-n}^{\infty}\left(\frac{m+k-n+1}{(m+k+1)^{2}}-\frac{n+k-m+1}{(n+k+1)^{2}}\right)\left|u_{k}\right|^{2}\right)\label{eq:Monomial Formula 1}
\end{align}

Now, 
\[
\frac{m+k-n+1}{(m+k+1)^{2}}-\frac{n+k-m+1}{(n+k+1)^{2}}=\frac{(n+k+1)^{2}\left(m+k-n+1\right)-(m+k+1)^{2}\left(n+k-m+1\right)}{(m+k+1)^{2}(n+k+1)^{2}}
\]
\begin{equation}
=\frac{(m^{2}-n^{2})k+\left(m-n+1\right)\left(n+1\right)^{2}+\left(m-n-1\right)\left(m+1\right)^{2}}{(m+k+1)^{2}(n+k+1)^{2}}.\label{eq:Monomial Formula 2}
\end{equation}
This is clearly positive when $k=m-n\geq1$.

Further, when we take the derivative of the numerator with respect
to $k$, we find that it is positive whenever $m>n$, and so the numerator
is increasing and thus always positive. Therefore we may conclude
that 
\[
\sum_{k=0}^{m-n-1}\frac{m+k-n+1}{(m+k+1)^{2}}\left|u_{k}\right|^{2}+\sum_{k=m-n}^{\infty}\left(\frac{m+k-n+1}{(m+k+1)^{2}}-\frac{n+k-m+1}{(n+k+1)^{2}}\right)\left|u_{k}\right|^{2}\geq0
\]
for all $u(z)=\sum_{k=0}^{\infty}u_{k}z^{k}\in A^{2}(\mathbb{D})$,
and so $T_{\varphi}$ is hyponormal. 
\end{proof}

\subsection{Non-harmonic polynomials}

We now turn to an examination of two term non-harmonic polynomials. 
\begin{thm}
\label{thm:Hyponormal + Hyponormal}Suppose $f=a_{m,n}z^{m}\bar{z}^{n}$
and $g=a_{i,j}z^{i}\bar{z}^{j}$, with $m>n$ and $i>j$. Then $T_{f+g}$
is hyponormal if for each $k\geq0$ the term
\[
\left|a_{m,n}\right|^{2}\frac{m-n+k+1}{(m+k+1)^{2}}+\left|a_{i,j}\right|^{2}\frac{i-j+k+1}{(i+k+1)^{2}}
\]
is sufficiently large. 
\end{thm}
\begin{rem*}
 Here sufficiently large means
that, under the assumption $m-n>i-j$, we have the following four conditions:
\[
\left|\frac{a_{m,n}}{a_{i,j}}\right|\frac{m+k-n+1}{(m+k+1)^{2}}+\left|\frac{a_{i,j}}{a_{m,n}}\right|\frac{i+k-j+1}{(i+k+1)^{2}}\geq C_{k}
\]
for $k\leq i-j-1$, and 
\[
\left|\frac{a_{m,n}}{a_{i,j}}\right|\frac{m+k-n+1}{(m+k+1)^{2}}+\left|\frac{a_{i,j}}{a_{m,n}}\right|\left(\frac{i+k-j+1}{(i+k+1)^{2}}-\frac{j+k-i+1}{(j+k+1)^{2}}\right)\geq C_{k}
\]
for $i-j\leq k\leq m-n-1$, and 
\[
\left|\frac{a_{m,n}}{a_{i,j}}\right|\left(\frac{m+k-n+1}{(m+k+1)^{2}}-\frac{n+k-m+1}{(n+k+1)^{2}}\right)+\left|\frac{a_{i,j}}{a_{m,n}}\right|\left(\frac{i+k-j+1}{(i+k+1)^{2}}-\frac{j+k-i+1}{(j+k+1)^{2}}\right)\geq C_{k}
\]
for $m-n\leq k\leq m-n+i-j-1,$ and 
\[
\left|\frac{a_{m,n}}{a_{i,j}}\right|\left(\frac{m+k-n+1}{(m+k+1)^{2}}-\frac{n+k-m+1}{(n+k+1)^{2}}\right)+\left|\frac{a_{i,j}}{a_{m,n}}\right|\left(\frac{i+k-j+1}{(i+k+1)^{2}}-\frac{j+k-i+1}{(j+k+1)^{2}}\right)\geq C_{k}+D_{k}
\]
where 
\begin{equation}
C_{k}:=\begin{cases}
\frac{m-n+k+1}{\left(m+k+1\right)\left(m-n+j+k+1\right)}, & \text{for }0\leq k\leq i-j-1,\\
\frac{m-n+k+1}{\left(m+k+1\right)\left(m-n+j+k+1\right)}-\frac{j-i+k+1}{\left(j+k+1\right)\left(j-i+m+k+1\right)}, & \text{for }k\geq i-j.
\end{cases}\label{Ck}
\end{equation}
and 
\begin{equation}
D_{k}:=\frac{j-i+k+1}{\left(j-i+n+k+1\right)\left(2j-i+k+1\right)}-\frac{2j-2i+n-m+k+1}{\left(2j-i+n-m+k+1\right)\left(2j-2i+n+k+1\right)}.\label{Dk}
\end{equation}
\end{rem*}
\begin{proof}
Recall that for $f,g\in L^{\infty}(\mathbb{D})$, and $u\in A^{2}$,
we have 
\begin{align}
\left\langle [T_{f+g}^{*},T_{f+g}]u,u\right\rangle =\left\Vert T_{f}u\right\Vert ^{2}-\left\Vert T_{f}^{*}u\right\Vert ^{2}+\left\Vert T_{g}u\right\Vert ^{2}-\left\Vert T_{g}^{*}u\right\Vert ^{2}+2\mathrm{Re}\left[\left\langle T_{f}u,T_{g}u\right\rangle -\left\langle T_{f}^{*}u,T_{g}^{*}u\right\rangle \right].\label{blahblahblah}
\end{align}

We begin to calculate the cross-term $2\mathrm{Re}\left[\left\langle T_{f}u,T_{g}u\right\rangle -\left\langle T_{f}^{*}u,T_{g}^{*}u\right\rangle \right]$.
Without loss of generality, we may assume that $m-n>i-j$. Under this
assumption, we find 
\begin{align*}
 & \,\,\quad2\mathrm{Re}\left[\left\langle T_{f}u,T_{g}u\right\rangle -\left\langle T_{f}^{*}u,T_{g}^{*}u\right\rangle \right]\\
 & =2\mathrm{Re}\left(a_{m,n}\overline{a_{i,j}}\right)\left[\left\langle \sum_{k=0}^{\infty}\frac{m+k-n+1}{m+k+1}u_{k}z^{m+k-n},\sum_{k=0}^{\infty}\frac{i+k-j+1}{i+k+1}\overline{u_{k}z^{i+k-j}}\right\rangle \right.\\
 & \qquad\qquad\qquad\qquad\left.-\left\langle \sum_{k=m-n}^{\infty}\frac{n+k-m+1}{n+k+1}u_{k}z^{n+k-m},\sum_{k=i-j}^{\infty}\frac{j+k-i+1}{j+k+1}\overline{u_{k}z^{j+k-i}}\right\rangle \right]\\
 & =2\sum_{k=0}^{\infty}C_{k}\mathrm{Re}\left(a_{m,n}\overline{a_{i,j}}u_{k}\overline{u_{k+m-n+i-j}}\right),
\end{align*}
where, for the purposes of slightly less daunting expressions, we
used $C_{k}$ as defined by \eqref{Ck} in the above remark. We will also, for reasons that will soon be clear, use $D_{k}$ as defined by \eqref{Dk}.

Unfortunately, as we have seen, we cannot control the sign of these
cross terms. Therefore, we will assume that we must always subtract
them. Further, by the inequality $2\mathrm{Re}\left(a\bar{b}\right)\leq\left|a\right|^{2}+\left|b\right|^{2}$,
we have 
\[
2\mathrm{Re}\left(a_{m,n}\overline{a_{i,j}}u_{k}\overline{u_{k+m-n+i-j}}\right)\leq\left|a_{m,n}a_{i,j}\right|\left(\left|u_{k}\right|^{2}+\left|u_{k+m-n+i-j}\right|^{2}\right).
\]

We combine equation \eqref{blahblahblah} with the calculations performed
in the proof of Theorem \ref{thm:Hyponormal monomials} to evaluate
$$\left\Vert T_{f}u\right\Vert ^{2}-\left\Vert T_{f}^{*}u\right\Vert ^{2}+\left\Vert T_{g}u\right\Vert ^{2}-\left\Vert T_{g}^{*}u\right\Vert ^{2}$$
applied to our given $f$ and $g$. Thereby we may conclude that $T_{\varphi}$
will be hyponormal if 
\begin{align*}
 & \quad\,\,\left|a_{m,n}\right|^{2}\left(\sum_{k=0}^{m-n-1}\frac{m+k-n+1}{(m+k+1)^{2}}\left|u_{k}\right|^{2}+\sum_{k=m-n}^{\infty}\left(\frac{m+k-n+1}{(m+k+1)^{2}}-\frac{n+k-m+1}{(n+k+1)^{2}}\right)\left|u_{k}\right|^{2}\right)\\
 & \quad\,\,+\left|a_{i,j}\right|^{2}\left(\sum_{k=0}^{i-j-1}\frac{i+k-j+1}{(i+k+1)^{2}}\left|u_{k}\right|^{2}+\sum_{k=i-j}^{\infty}\left(\frac{i+k-j+1}{(i+k+1)^{2}}-\frac{j+k-i+1}{(j+k+1)^{2}}\right)\left|u_{k}\right|^{2}\right)\\
 & \ge\left|a_{m,n}a_{i,j}\right|\sum_{k=0}^{\infty}C_{k}\left(\left|u_{k}\right|^{2}+\left|u_{k+m-n+i-j}\right|^{2}\right)\\
 & =\sum_{k=0}^{\infty}C_{k}\left|u_{k}\right|^{2}+\sum_{k=m-n+i-j}^{\infty}D_{k}\left|u_{k}\right|^{2}.
\end{align*}

Thus, an appropriate term by term comparison of the coefficients of
$\left|u_{k}\right|^{2}$ will show that operator $T_{f+g}$ is hyponormal,
if the bounds given in the above remark hold.

In particular, we obtain the stronger estimate
\[
\left\langle [T_{f+g}^{*},T_{f+g}]u,u\right\rangle \ge\sum_{k=0}^{\infty}A_{k}\left|u_{k}\right|^{2},
\]
where $A_{k}$ is non-negative for all $k$. 
\end{proof}
The next theorem examines the case when one of the terms in our binomial
is the symbol of a cohyponormal operator (i.e.~an operator whose adjoint is hyponormal). This is in contrast to Theorem \ref{thm:Hyponormal + Hyponormal}
where each term individually yielded a hyponormal operator. 
\begin{thm}
\label{thm:Hyponormal + Co-Hyponormal}Suppose $f=a_{m,n}z^{m}\bar{z}^{n}$
and $g=a_{i,j}\bar{z}^{i}z^{j}$, with $m>n$ and $i>j$. Then $T_{f+g}$
is hyponormal if for each $k\geq0$ 
\[
\left|a_{m,n}\right|^{2}\frac{m-n+k+1}{(m+k+1)^{2}}-\left|a_{i,j}\right|^{2}\frac{i-j+k+1}{(i+k+1)^{2}}
\]
is sufficiently large. 
\end{thm}
\begin{rem}\label{rem:Hyponormal + Co-Hyponormal}
Here, as in Theorem \ref{thm:Hyponormal + Hyponormal}, we can specify what sufficiently
large means. To do so, we abbreviate
\begin{align}
\widetilde{A}_{k} & :=\left|a_{m,n}\right|^{2}\left(\frac{m+k-n+1}{(m+k+1)^{2}}-\frac{n+k-m+1}{(n+k+1)^{2}}\right),\text{ and }\label{wtAk}
\\
\widetilde{B}_{k} & :=\left|a_{i,j}\right|^{2}\left(\frac{i+k-j+1}{(i+k+1)^{2}}-\frac{j+k-i+1}{(j+k+1)^{2}}\right),
\label{wtBk}
\end{align}
as well as
\begin{align}
\widetilde{C}_{k}&:=\frac{m-n+k+1}{(m+k+1)(m-n+i+k+1)}-\frac{i-j+k+1}{(i+k+1)(i-j+m+k+1)}\,,
\text{ and }\label{wtCk}
\\
\widetilde{D}_{k}&:=\frac{k+j-i+1}{(n+j-i+k+1)(j+k+1)}-\frac{k+n-m+1}{(m-n+j+k+1)(n+k+1)}\,.
\label{wtDk}
\end{align}
Now sufficiently large means that the following four conditions are satisfied:
\[
\left|\frac{a_{m,n}}{a_{i,j}}\right|\frac{m+k-n+1}{(m+k+1)^{2}}-\left|\frac{a_{i,j}}{a_{m,n}}\right|\frac{i+k-j+1}{(i+k+1)^{2}}\geq\widetilde{C}_{k}
\]
for $k\leq\mathrm{min}\{m-n,i-j\}-1$, and 
\[
\widetilde{C}_{k}\leq\begin{cases}
\left|\frac{a_{m,n}}{a_{i,j}}\right|\frac{m+k-n+1}{(m+k+1)^{2}}-\frac{\widetilde{B}_{k}}{\left|a_{m,n}a_{i,j}\right|} \quad& \text{when }m-n>i-j\par\vspace{0.2in}\\
\frac{\widetilde{A}_{k}}{\left|a_{m,n}a_{i,j}\right|}-\left|\frac{a_{i,j}}{a_{m,n}}\right|\frac{i+k-j+1}{(i+k+1)^{2}} & \text{when }m-n<i-j
\end{cases}
\]
for $\mathrm{min}\{m-n,i-j\}\leq k\leq\mathrm{max}\{m-n,i-j\}-1$,
and 
\[
\frac{\widetilde{A}_{k}-\widetilde{B}_{k}}{\left|a_{m,n}a_{i,j}\right|}\geq\widetilde{C}_{k}
\]
for $\mathrm{max}\{m-n,i-j\}\leq k\leq m-n+i-j-1$, and 
\[
\frac{\widetilde{A}_{k}-\widetilde{B}_{k}}{\left|a_{m,n}a_{i,j}\right|}
\geq\widetilde{C}_{k}+\widetilde{D}_{k}
\]
for $k\geq m-n+i-j$. 
\end{rem}
\begin{proof}
Recall that for $f,g\in L^{\infty}(\mathbb{D})$, and $u\in A^{2}(\mathbb{D})$,
we have 
\[
\left\langle [T_{f+g}^{*},T_{f+g}]u,u\right\rangle =\left\Vert T_{f}u\right\Vert ^{2}-\left\Vert T_{f}^{*}u\right\Vert ^{2}+\left\Vert T_{g}u\right\Vert ^{2}-\left\Vert T_{g}^{*}u\right\Vert ^{2}+2\mathrm{Re}\left[\left\langle T_{f}u,T_{g}u\right\rangle -\left\langle T_{f}^{*}u,T_{g}^{*}u\right\rangle \right].
\]

Again, the calculations performed in the proof of Theorem \ref{thm:Hyponormal monomials}
applied to the current $f$ and $g$ show 
\begin{align}
 & \quad\,\,\left\Vert T_{f}u\right\Vert ^{2}-\left\Vert T_{f}^{*}u\right\Vert ^{2}+\left\Vert T_{g}u\right\Vert ^{2}-\left\Vert T_{g}^{*}u\right\Vert ^{2}\nonumber \\
 & =\left|a_{m,n}\right|^{2}\sum_{k=0}^{m-n-1}\frac{m+k-n+1}{(m+k+1)^{2}}\left|u_{k}\right|^{2}+\sum_{k=m-n}^{\infty}\widetilde{A}_{k}\left|u_{k}\right|^{2}\\
 & \quad\,\,-\left|a_{i,j}\right|^{2}\sum_{k=0}^{i-j-1}\frac{i+k-j+1}{(i+k+1)^{2}}\left|u_{k}\right|^{2}-\sum_{k=i-j}^{\infty}\widetilde{B}_{k}\left|u_{k}\right|^{2},
\end{align}
where we used $\widetilde{A}_{k}$ and $\widetilde{B}_{k}$ as defined in \eqref{wtAk} and \eqref{wtBk}.

This proof differs from that of Theorem \ref{thm:Hyponormal + Hyponormal}
in the calculation for the cross-terms.
Under the assumption $i>j$, we have 
\begin{align}
 & \,\,\quad2\mathrm{Re}\left[\left\langle T_{f}u,T_{g}u\right\rangle -\left\langle T_{f}^{*}u,T_{g}^{*}u\right\rangle \right]\nonumber \\
 & =2\mathrm{Re}\left(a_{m,n}\overline{a_{i,j}}\right)\left[\left\langle \sum_{k=0}^{\infty}\frac{m+k-n+1}{m+k+1}u_{k}z^{m+k-n},\sum_{k=i-j}^{\infty}\frac{j+k-i+1}{j+k+1}\overline{u_{k}z^{j+k-i}}\right\rangle \right.\nonumber \\
 & \qquad\qquad\qquad\qquad\left.-\left\langle \sum_{k=m-n}^{\infty}\frac{n+k-m+1}{n+k+1}u_{k}z^{n+k-m},\sum_{k=i-j}^{\infty}\frac{i+k-j+1}{i+k+1}\overline{u_{k}z^{i+k-j}}\right\rangle \right]\nonumber \\
 & =2\mathrm{Re}\left(a_{m,n}\overline{a_{i,j}}\right)\sum_{k=0}^{\infty}\widetilde{C}_{k}u_{k}\overline{u_{m-n+i-j+k}}\,.\label{eq:Cross Terms chunk}
\end{align}
via direct calculation and with $\widetilde{C}_{k}$ from \eqref{wtCk}.

The argument now follows \emph{mutatis mutandis} as in Theorem \ref{thm:Hyponormal + Hyponormal}.
In particular, with $\widetilde{D}_{k}$ from \eqref{wtDk} 
and once again taking advantage of the inequality $2\mathrm{Re}\left(a\bar{b}\right)\leq\left|a\right|^{2}+\left|b\right|^{2}$,
we have that if the conditions given in Remark \ref{rem:Hyponormal + Co-Hyponormal} hold, then operator
$T_{f+g}$ will be hyponormal. 
\end{proof}
Both of the above theorems are rather cumbersome to apply directly.
Further, it is not immediately clear \emph{a priori} that the relevant
bounds are ever actually attainable. In the following example we look
at a symbol which shows that the bounds in Theorem \ref{thm:Hyponormal + Co-Hyponormal}
can be attained. This shows that while a seemingly ``nice'' symbol
like $T_{z-3\left|z\right|^{2}}$ might fail to be hyponormal even
though it is the sum of a sub-normal operator and a self-adjoint operator,
the sum of a hyponormal and co-hyponormal operator might still produce
an operator which is hyponormal. 
\begin{example}
Consider $\varphi(z)=z^{2}\overline{z}+\frac{1}{7}\bar{z}^{4}z^{3}$.
We can plug this into the relevant calculations from Theorem \ref{thm:Hyponormal + Co-Hyponormal}
to test for hyponormality. In particular, we find that 
\[
\widetilde{A}_{k}=\frac{3k+8}{(k+3)^{2}(k+2)^{2}},
\]
and that 
\[
\widetilde{B}_{k}+\left|a_{m,n}a_{i,j}\right|\left(\widetilde{C}_{k}+\widetilde{D}_{k}\right)
\]
\begin{equation}
=\frac{1}{7}\left(\frac{7k+32}{7(k+5)^{2}(k+4)^{2}}+\frac{3k^{3}+21k^{2}+46k+8}{(k+6)(k+5)(k+4)(k+3)(k+2)(k+1)}\right).\label{eq:Rational Bound}
\end{equation}
Thus, we find that $T_{\varphi}$ will be hyponormal if 
\begin{align*}
 & \widetilde{A}_{k}-\widetilde{B}_{k}-\left|a_{m,n}a_{i,j}\right|\left(\widetilde{C}_{k}+\widetilde{D}_{k}\right)\\
= & \frac{3k+8}{(k+3)^{2}(k+2)^{2}}-\frac{1}{7}\left(\frac{7k+32}{7(k+5)^{2}(k+4)^{2}}+\frac{3k^{3}+21k^{2}+46k+8}{(k+6)(k+5)(k+4)(k+3)(k+2)(k+1)}\right)\\
= & \frac{119k^{7}+3475k^{6}+41785k^{5}+267977k^{4}+985764k^{3}+2061168k^{2}+2228760k+927168}{49(k+6)(k+5)^{2}(k+4)^{2}(k+3)^{2}(k+2)^{2}(k+1)}>0
\end{align*}
for all $k\geq2$, since the checks for $k=0,1$ show the desired
inequalities hold.

However in fact, it is clear from observation that this rational function
is positive for all $k>0$, and in particular for $k\geq2$. Thus
$T_{\varphi}$ is hyponormal. This example will be explored more in depth
in Theorem \ref{thm: Hyponormal + Co-Hyponormal construction}.

Note however that for our choice of $\varphi,$ since we have that
the expression in \eqref{eq:Rational Bound} is less than $\frac{3k+8}{(k+3)^{2}(k+2)^{2}}$
for all $k>0$, only one check was actually necessary to show that
$T_{\varphi}$ is hyponormal. Indeed the construction of this example
was based on ensuring a sufficiently quick decay of the expression in \eqref{eq:Rational Bound}
while also ensuring that for small values of $k$ the required inequalities
would still hold. In the following theorem, we generalize the idea
of this construction to find a general construction for hyponormal
operators whose symbol is of the form in the hypothesis of Theorem
\ref{thm:Hyponormal + Co-Hyponormal}.
\end{example}
\begin{thm}
\label{thm: Hyponormal + Co-Hyponormal construction}Fix $\delta\in\mathbb{N}.$
For every integer $n\in\mathbb{N}$ there exists $j\in\mathbb{N}$,
such that $T_{\varphi}$ with symbol $\varphi(z)=z^{n+\delta}\overline{z}^{n}+\frac{1}{2j+\delta}\overline{z}^{j+\delta}z^{j}$
is hyponormal. 
\end{thm}
\begin{proof}
The idea of the proof lies in specifying the words ``sufficiently
large\textquotedbl{} in Theorem \ref{thm:Hyponormal + Co-Hyponormal} in accordance with Remark \ref{rem:Hyponormal + Co-Hyponormal}.

We let $m=n+\delta$ and $i=j+\delta$. Since $m-n=i-j=\delta$, the
formulas from Remark \ref{rem:Hyponormal + Co-Hyponormal}
become somewhat simplified. Recall that we use $\widetilde{A}_{k}=\left|a_{m,n}\right|^{2}\left(\frac{m+k-n+1}{(m+k+1)^{2}}-\frac{n+k-m+1}{(n+k+1)^{2}}\right),$
as well as $\widetilde{B}_{k}=\left|a_{i,j}\right|^{2}\left(\frac{i+k-j+1}{(i+k+1)^{2}}-\frac{j+k-i+1}{(j+k+1)^{2}}\right).$
In particular, for $k\ge\delta$ and with $a_{m,n}=1$, we can arrive
at
\[
\widetilde{A}_{k}=\frac{(m+n)\delta k+(\delta+1)(n+1)^{2}+(\delta-1)(m+1)^{2}}{(k+m+1)^{2}(k+n+1)^{2}}.
\]
And, with $a_{i,j}=\frac{1}{i+j}$ we obtain 
\[
\widetilde{B}_{k}=\frac{(i+j)\delta k+(\delta+1)(j+1)^{2}+(\delta-1)(j+1)^{2}}{(i+j)^{2}(k+i+1)^{2}(k+j+1)^{2}}.
\]
Finally, we have 
\begin{align*}
\widetilde{C}_{k} & =\frac{\delta(i-m)(k+\delta+1)}{(k+m+1)(k+m+\delta+1)(k+i+1)(k+i+\delta+1)}\,,\text{ and}\\
\widetilde{D}_{k} & =\frac{\delta(i-m)(k-\delta+1)}{(k+m+1)(k+n-\delta+1)(k+j+1)(k+j-\delta+1)}\,.
\end{align*}

Recall that our aim is now to prove that for $k\ge2\delta$ we have
\begin{equation}
(i+j)\widetilde{A}_{k}\ge(i+j)\widetilde{B}_{k}+\widetilde{C}_{k}+\widetilde{D}_{k},\label{eq:Symbolic Inequality}
\end{equation}
since $a_{i,j}=\frac{1}{i+j}.$ This is a direct application of the
bounds given in Theorem \ref{thm:Hyponormal + Co-Hyponormal}.

Our goal will be to prove that the numerator of $(i+j)\widetilde{A}_{k}$
is larger than the sums of the numerators of $(i+j)\widetilde{B}_{k}$,
$\widetilde{C}_{k}$, and $\widetilde{D}_{k}$, while ensuring that
the denominator of $\widetilde{A}_{k}$ is smaller than each of the
denominators of $(i+j)\widetilde{B}_{k}$, $\widetilde{C}_{k}$, and
$\widetilde{D}_{k}$. If we can show this we will have shown that
\eqref{eq:Symbolic Inequality} holds for all $k\ge2\delta$, and
in fact, the other required bounds of Theorem \ref{thm:Hyponormal + Co-Hyponormal}
will also necessarily follow immediately, guaranteeing the hyponormality
of $T_{\varphi}$.

Looking first at the numerators then, we first wish to show 
\begin{equation}
(i+j)(m+n)\delta k\ge(2i-2m+1)k,\label{eq:Numerator k-terms}
\end{equation}
for all $k\ge2\delta$. Yet since clearly $(i+j)(m+n)\delta>(2i-2m+1)$,
we have that \eqref{eq:Numerator k-terms} holds for all $k\ge0$.
Looking at the constant terms of the numerators, and multiplying through
by $(i+j)$ to prevent a fraction in the constant term of the numerator
$\widetilde{B}_{k}$, it is equally clear that 
\begin{equation}
(i+j)^{2}\left[(\delta+1)(n+1)^{2}+(\delta-1)(m+1)^{2}\right]\ge(\delta+1)(j+1)^{2}+(\delta-1)(j+1)^{2}+2\delta(i-m)(i+j),\label{eq:Numerator constant terms}
\end{equation}
since the inequality 
\[
(i+j)^{2}\left[(\delta+1)+(\delta-1)\right]\ge(\delta+1)(j+1)^{2}+(\delta-1)(j+1)^{2},
\]
and the inequality 
\[
(i+j)^{2}\left[(\delta+1)(n^{2}+2n)+(\delta-1)(m^{2}+2m)\right]\ge2\delta(i-m)(i+j)
\]
both hold by inspection. So we have that the numerator of $\widetilde{A}_{k}$
is larger than the sums of the numerators of $(i+j)\widetilde{B}_{k}$,
$\widetilde{C}_{k}$, and $\widetilde{D}_{k}$, as desired.

It remains to show our desired inequalities for the denominators.
It is clear by inspection that if $j>m$, then we have that 
\[
(k+m+1)^{2}(k+n+1)^{2}\le(k+i+1)^{2}(k+j+1)^{2}
\]
and 
\[
(k+m+1)^{2}(k+n+1)^{2}\le(k+m+1)(k+m+\delta+1)(k+i+1)(k+i+\delta+1).
\]

We take a moment to show that it is possible to choose $j$ large
enough so that 
\begin{align}
(k+m+1)^{2}(k+n+1)^{2}<(k+m+1)(k+n-\delta+1)(k+j+1)(k+j-\delta+1)\label{blehblehbleh}
\end{align}
for all $k\ge2\delta$. Since we have already assumed that $j>m$, we have that $j-\delta>n$, and thus \eqref{blehblehbleh}
follows so long as 
\[
(k+m+1)(k+n+1)<(k+n-\delta+1)(k+j+1).
\]
Or equivalently, since $k\ge2\delta$, inequality \eqref{blehblehbleh}
follows so long as 
\[
j>\frac{(k+m+1)(k+n+1)}{k+n-\delta+1}-k-1=\frac{k(m+\delta)+mn+m+\delta}{k+n-\delta+1}=:q(k).
\]
Since the rational function $q(k)$ remains bounded for $k\in[2\delta,\infty)$,
it is possible to choose an appropriate $j\in\mathbb{N}$. Thus \eqref{eq:Symbolic Inequality}
holds for all $k\ge2\delta$.

The same argument will show that $(i+j)\widetilde{A}_{k}\ge(i+j)\widetilde{B}_{k}+\widetilde{C}_{k}$
holds for $\delta\le k\le2\delta$. The required bounds for $k<\delta$
hold trivially.

Thus, by Theorem \ref{thm:Hyponormal + Co-Hyponormal}, operator $T_{\varphi}$
is hyponormal.
\end{proof}

\section{Polynomials of Fixed Relative Degree\label{sec:Polynomials of Fixed Relative Degree}}

We now turn to operators whose symbol is a polynomial of the form
$$
\varphi(z)=a_{1}z^{m_{1}}\bar{z}^{n_{1}}+\ldots+a_{k}z^{m_{k}}\bar{z}^{n_{k}},\qquad
\text{with }m_{1}-n_{1}=\ldots=m_{k}-n_{k}=\delta\geq0.
$$
We shall call
these \emph{polynomials of fixed relative degree.} Though working
with non-harmonic symbols can be difficult, some results are known
in these special cases. One which we will be interested in for this
paper is due to Y.~Liu and C.~Lu in \cite[Theorem 3.1]{Liu-Lu}. There they
make use of the \emph{Mellin transform }of $\varphi.$
\begin{defn*}
Suppose $\varphi\in L^{1}\left(\left[0,1\right],rdr\right)$. For
$\mathrm{Re}\:z\geq2$, the \emph{Mellin transform} of $\varphi$,
is given by 
\[
\widehat{\varphi}(z):=\int_{0}^{1}\varphi(x)x^{z-1}dx.
\]
\end{defn*}
For $\varphi(re^{i\theta})=e^{ik\theta}\varphi_{0}(r),$ with $k\in\mathbb{Z}$
and $\varphi_{0}$ radial, we can compute the action of $T_{\varphi}$
on $z^{n}$. Specifically, 
\[
T_{\varphi}z^{n}=\begin{cases}
2\left(n+k+1\right)\widehat{\varphi}_{0}(2n+k+2)z^{n+k} \quad& n+k\geq0\\
0 & n+k<0,
\end{cases}
\]
and 
\[
T_{\bar{\varphi}}z^{n}=\begin{cases}
2\left(n-k+1\right)\widehat{\varphi}_{0}(2n-k+2)z^{n-k} \quad& n-k\geq0\\
0 & n-k<0.
\end{cases}
\]
Using this, Y.~Liu and C.~Lu proved the following theorem in \cite[Theorem 3.1]{Liu-Lu}. 
\begin{thm}
\label{thm:Liu-Lu Theorem}Let $\varphi(re^{i\theta})=e^{i\delta\theta}\varphi_{0}(r)\in L^{\infty}(\mathbb{D}),$
where $\delta\in\mathbb{Z}$ and $\varphi_{0}$ is radial. Then $T_{\varphi}$
is hyponormal if and only if one of the following conditions holds: 
\begin{itemize}
\item[1)] $\delta=0$ and $\varphi_{0}\equiv0$; 
\item[2)] $\delta=0$; 
\item[3)] $\delta>0$ and for each $\alpha\geq\delta$, 
\[
\left|\widehat{\varphi}_{0}(2\alpha+\delta+2)\right|\geq c_{\alpha,\delta}\left|\widehat{\varphi}_{0}(2\alpha-\delta+2)\right|,
\]
where we abbreviate 
\[
c_{\alpha,\delta}:=\sqrt{\frac{\alpha-\delta+1}{\alpha+\delta+1}}.
\]
\end{itemize}
\end{thm}
The first situation immediately implies that if $\varphi(z)$ is a
polynomial in $z$ and $\bar{z}$ where the degree of $\bar{z}$ is
larger than the degree of $z$ in each term, then $T_{\varphi}$ cannot
be hyponormal. The second situation is a consequence of the fact that
whenever $\varphi$ is real valued in $\mathbb{D}$, then $T_{\varphi}$
is actually self-adjoint and thus trivially hyponormal. The final
situation, when $\delta>0$, will be of interest to us. 
\begin{rem*}
One can prove Theorem \ref{thm:Hyponormal monomials} by applying
Theorem \ref{thm:Liu-Lu Theorem}, however the proof is non-trivial.
\end{rem*}
The following is a corollary of Theorem \ref{thm:Same Degree Polynomial2}.
However, a direct proof is simple enough that we showcase it here
for the convenience of the reader. 
\begin{cor}
\label{thm:Same degree polynomial1}Let $\varphi(z)=a_{1}z^{m_{1}}\bar{z}^{n_{1}}+\ldots+a_{k}z^{m_{k}}\bar{z}^{n_{k}}$,
with $m_{1}-n_{1}=\ldots=m_{k}-n_{k}=\delta\geq0$, and $a_{i}$ all
lying along the same ray for $1\leq i\leq k$ (i.e. $\mathrm{arg}(a_{1})=\ldots=\mathrm{arg}\left(a_{k}\right)$),
then $T_{\varphi}$ is hyponormal. 
\end{cor}
\begin{proof}
Write $\varphi=\varphi_{1}+\ldots+\varphi_{k}$, where $\varphi_{i}=a_{i}e^{i\delta\theta}\varphi_{0,i}(r)$.
Recall that $c_{\alpha,\delta}=\sqrt{\frac{\alpha-\delta+1}{\alpha+\delta+1}}$.
By Theorem \ref{thm:Liu-Lu Theorem} and Theorem \ref{thm:Hyponormal monomials},
we have that for each $\alpha\geq\delta$ 
\[
\left|a_{i}\widehat{\varphi}_{0,i}(2\alpha+\delta+2)\right|\geq c_{\alpha,\delta}\left|a_{i}\widehat{\varphi}_{0,i}(2\alpha-\delta+2)\right|.
\]

Since the $a_{i}'s$ all lie along the same ray, we have that for
each $n\geq\delta$ 
\[
\left|\sum_{i=1}^{k}a_{i}\widehat{\varphi}_{0,i}(2\alpha+\delta+2)\right|=\sum_{i=1}^{k}\left|a_{i}\right|\left|\widehat{\varphi}_{0,i}(2\alpha+\delta+2)\right|\geq\sum_{i=1}^{k}c_{\alpha,\delta}\left|a_{i}\right|\left|\widehat{\varphi}_{0,i}(2\alpha-\delta+2)\right|.
\]
The claim now follows by Theorem \ref{thm:Liu-Lu Theorem}. 
\end{proof}
One is tempted to conjecture that the argument of these coefficients
should not matter. However the following example shows that this is
not the case. 
\begin{example}
 Let $\varphi(z)=z^{2}\bar{z}-z^{3}\bar{z}^{2}$. Then $\widehat{\varphi}_{0}(k)=\frac{1}{k+3}-\frac{1}{k+5}$,
and we find that 
\[
\frac{1}{2\alpha+6}-\frac{1}{2\alpha+8}<\sqrt{\frac{\alpha}{\alpha+2}}\left(\frac{1}{2\alpha+4}-\frac{1}{2\alpha+6}\right),
\]
whenever $\alpha\geq2.$ This violates the conditions of Theorem \ref{thm:Liu-Lu Theorem},
and so $T_{\varphi}$ cannot be hyponormal.
 
\end{example}
So, can we find sufficient conditions, beyond all coefficients lying
along the same ray, to guarantee that such functions yield hyponormal
operators? The answer is yes, and depends somewhat on the number of
terms, as well as the relative position of the coefficients, as the following two theorems demonstrate. 
\begin{thm}
\label{thm:Same Relative Degree Binomial}Let $\varphi(z)=a_{1}z^{m}\bar{z}^{n}+a_{2}z^{i}\bar{z}^{j}$,
with $m-n=i-j=\delta\geq0$. Then $T_{\varphi}$ is hyponormal if
$a_{1}$ and $a_{2}$ lie in the same quarter-plane (i.e.~$\left|\mathrm{arg}\left(a_{1}\right)-\mathrm{arg}\left(a_{2}\right)\right|\leq\frac{\pi}{2}$).
Further, under the additional condition that
\begin{align}\label{Condition}
0\le\frac{\left|a_{1}\right|}{\alpha+m+1}-\frac{\left|a_{2}\right|}{\alpha+i+1}<c_{\alpha,\delta}^{2}\left(\frac{\left|a_{1}\right|}{\alpha+n+1}-\frac{\left|a_{2}\right|}{\alpha+j+1}\right)
\qquad
\text{for all }\alpha,
\end{align}
the requirement that $\left|\mathrm{arg}\left(a_{1}\right)-\mathrm{arg}\left(a_{2}\right)\right|\leq\frac{\pi}{2}$
is also necessary for the hyponormality of $T_{\varphi}.$
\end{thm}
\begin{proof}
We begin with some general observations.
Without loss of generality, we may assume that $a_{1}$ is a positive
real number and that $a_{2}=r_{2}e^{i\theta}$ with 
$\frac{-\pi}{2}\leq\theta\le\frac{\pi}{2}$. We have 
$\widehat{\varphi}_{0}(k)=\frac{a_{1}}{m+n+k}+\frac{a_{2}}{i+j+k}$. 
Recall that $c_{\alpha,\delta}=\sqrt{\frac{\alpha-\delta+1}{\alpha+\delta+1}}$.
By Theorem \ref{thm:Liu-Lu Theorem}, $T_{\varphi}$ will be hyponormal
if and only if 
\begin{align}
 & \left|\widehat{\varphi}_{0}(2\alpha+\delta+2)\right|^{2}\geq c_{\alpha,\delta}^{2}\left|\widehat{\varphi}_{0}(2\alpha-\delta+2)\right|^{2},\nonumber
 \end{align}
which is equivalent to
\begin{align*}
& \left|\frac{a_{1}}{\alpha+m+1}+\frac{a_{2}}{\alpha+i+1}\right|^{2}\geq c_{\alpha,\delta}^{2}\left|\frac{a_{1}}{\alpha+n+1}+\frac{a_{2}}{\alpha+j+1}\right|^{2}, 
\end{align*}
as well as to
\begin{align}
 &\quad \left(\frac{a_{1}}{\alpha+m+1}+\frac{r_{2}\cos\theta}{\alpha+i+1}\right)^{2}+\frac{r_{2}^{2}\sin^{2}\theta}{\alpha+i+1}\nonumber \\
&\ge  c_{\alpha,\delta}^{2}\left[\left(\frac{a_{1}}{\alpha+n+1}+\frac{r_{2}\cos\theta}{\alpha+j+1}\right)^{2}+\frac{r_{2}^{2}\sin^{2}\theta}{\left(\alpha+j+1\right)^{2}}\right],\label{eq:Two-term bound}
\end{align}
for all $\alpha\ge\delta$.

Let us focus on proving the first statement.
By the hypothesis that $i=\delta+j$, we can verify
\[
\frac{r_{2}^{2}\sin^{2}\theta}{\left(\alpha+i+1\right)^{2}}\ge c_{\alpha,\delta}^{2}\frac{r_{2}^{2}\sin^{2}\theta}{\left(\alpha+j+1\right)^{2}}
\]
for all $\alpha\ge\delta$. Similarly, 
\[
\left(\frac{a_{1}}{\alpha+m+1}+\frac{r_{2}\cos\theta}{\alpha+i+1}\right)^{2}\ge c_{\alpha,\delta}^{2}\left(\frac{a_{1}}{\alpha+n+1}+\frac{r_{2}\cos\theta}{\alpha+j+1}\right)^{2}
\]
so long as $\cos\theta\ge0$. That is, when $a_{2}$ is in
the closed right half-plane. Thus, it follows that when $\left|\mathrm{arg}\left(a_{1}\right)-\mathrm{arg}\left(a_{2}\right)\right|\leq\frac{\pi}{2}$,
then the estimate in equation \eqref{eq:Two-term bound} holds for all $\alpha\geq\delta$.
And so $T_{\varphi}$ is hyponormal by Theorem \ref{thm:Liu-Lu Theorem}. 

To show the converse, we assume the extra condition \eqref{Condition}. We will show that if $\frac{\pi}{2}<\theta\le\pi,$
then there exists an $\alpha$ for which \eqref{eq:Two-term bound}
fails, and consequently $T_{\varphi}$ must fail to be hyponormal
by Theorem \ref{thm:Liu-Lu Theorem}.

First, fix $\alpha\ge\delta.$ We construct two circles. $$C_{1}:=\left\{ z:\;\left|z-\frac{a_{1}}{\alpha+m+1}\right|=\frac{r_{2}}{\alpha+i+1}\right\}, $$
centered at $\frac{a_{1}}{\alpha+m+1}$ with radius $\frac{r_{2}}{\alpha+i+1}$,
and $$C_{2}:=\left\{ z:\;\left|z-c_{\alpha,\delta}^{2}\frac{a_{1}}{\alpha+n+1}\right|=c_{\alpha,\delta}^{2}\frac{r_{2}}{\alpha+j+1}\right\}, $$
centered at $c_{\alpha,\delta}^{2}\frac{a_{1}}{\alpha+n+1}$ with
radius $c_{\alpha,\delta}^{2}\frac{r_{2}}{\alpha+j+1}.$ Without loss of generality, we may always assume that both of these circles lie in the right half-plane.

So long as the difference of their centers is bounded by the difference of their radii, i.e.
$$
\frac{a_{1}}{\alpha+m+1}-c_{\alpha,\delta}^{2}\frac{a_{1}}{\alpha+n+1}<\frac{r_{2}}{\alpha+i+1}-c_{\alpha,\delta}^{2}\frac{r_{2}}{\alpha+j+1},
$$
we have that $C_{2}$ lies completely in the region bounded by $C_{2}$. Such a scenario is illustrated in Figure \ref{FigInside} for one value of $\alpha = 6$.

\begin{figure}[h]
    \includegraphics[scale=0.6]{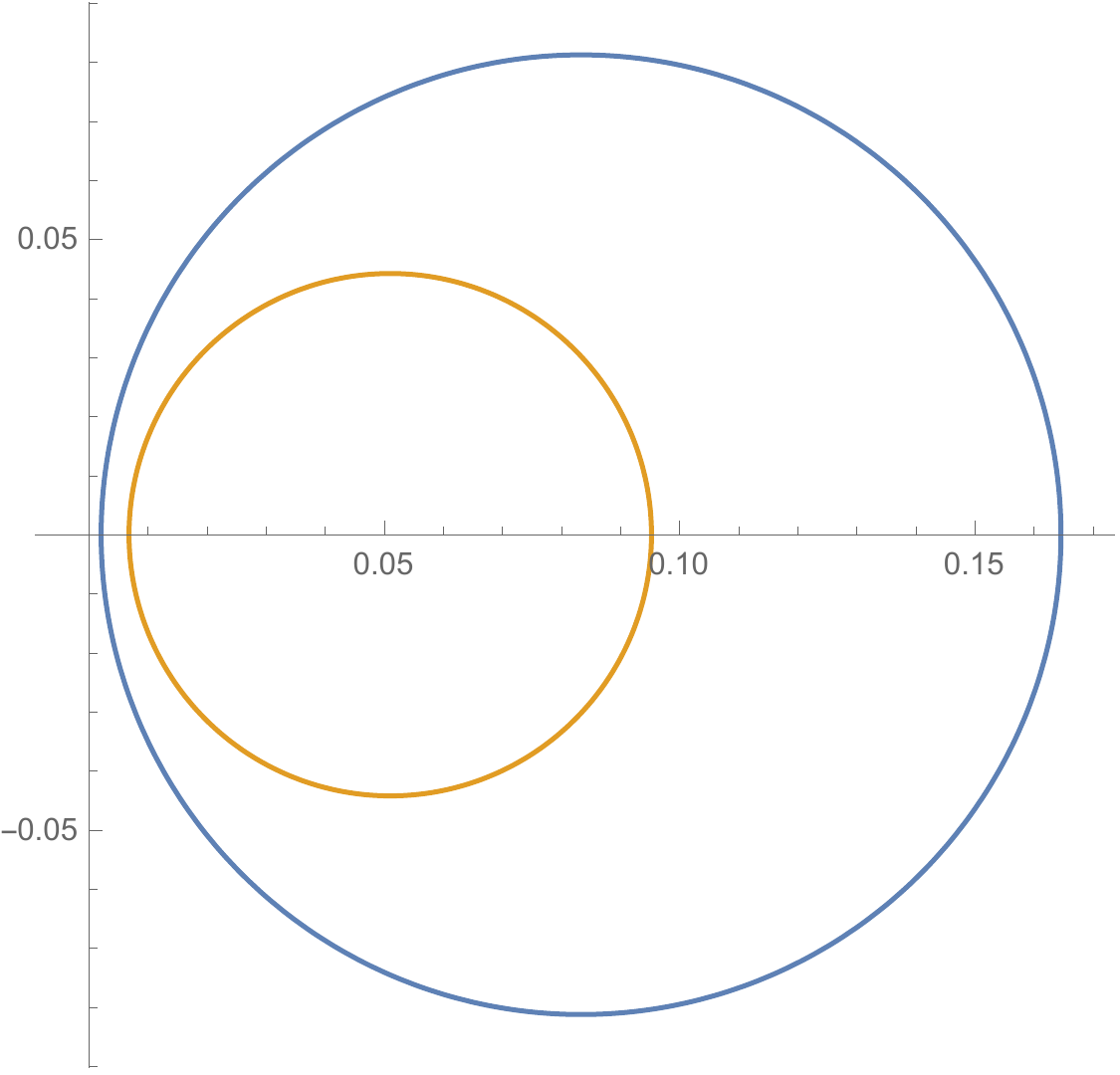}
    \caption{The situation when $\alpha=6$, $m=5$, $i=9$, and $\delta=4$.}
\label{FigInside}
\end{figure}

In this case, it is clear that there exists a $\frac{\pi}{2}<\theta<\pi$
such that
\begin{equation}
\left(\frac{a_{1}}{\alpha+m+1}+\frac{r_{2}\cos\theta}{\alpha+i+1}\right)^{2}-c_{\alpha,\delta}^{2}\left(\frac{a_{1}}{\alpha+n+1}+\frac{r_{2}\cos\theta}{\alpha+j+1}\right)^{2}=0.\label{eq:RealPart}
\end{equation}
Then if $\theta\rightarrow\pi$, the left hand side of \eqref{eq:RealPart}
will converge to a negative real number by condition \eqref{Condition}. At the same time, since 
\[
\lim_{\theta\rightarrow\pi}
\left(
\frac{r_{2}^{2}\sin^{2}\theta}{\left(\alpha+i+1\right)^{2}}-\frac{r_{2}^{2}\sin^{2}\theta}{\left(\alpha+j+1\right)^{2}}
\right)
=0,
\]
there exists some $\theta$ for which \eqref{eq:Two-term bound} fails.
Define 
\[
\theta_{\alpha}:=\inf\left\{ \theta:\text{ equation }\eqref{eq:Two-term bound}\;\mathrm{fails}\right\} .
\]

 We will now show that $\theta_{\alpha}\rightarrow\frac{\pi}{2}$
as $\alpha\rightarrow\infty$. Define 
\[
F_{\alpha}(\theta):=\left(\frac{a_{1}}{\alpha+m+1}+\frac{r_{2}\cos\theta}{\alpha+i+1}\right)^{2}+\frac{r_{2}^{2}\sin^{2}\theta}{\left(\alpha+i+1\right)^{2}}-c_{\alpha,\delta}^{2}\left[\left(\frac{a_{1}}{\alpha+n+1}+\frac{r_{2}\cos\theta}{\alpha+j+1}\right)^{2}+\frac{r_{2}^{2}\sin^{2}\theta}{\left(\alpha+j+1\right)^{2}}\right].
\]
As shown above, there exists a $\theta$ such that $F_{\alpha}(\theta)=0$.
It must be the case that $\theta = \theta_{\alpha}$ is a root, since
$F_{\alpha}(\theta)>0$ for $\theta<\theta_{\alpha}$,
and since $F_{\alpha}(\theta)<0$ for $\theta>\theta_{\alpha}.$
Solving for this $\theta_{\alpha}$, we find that 
\begin{align*}
\theta_{\alpha}= & \arccos\left(\frac{c_{\alpha,\delta}^{2}\left[\frac{1}{\left(\alpha+n+1\right)^{2}}+\frac{r_{2}^{2}}{\left(\alpha+j+1\right)^{2}}\right]-\left[\frac{1}{\left(\alpha+m+1\right)^{2}}+\frac{r_{2}^{2}}{\left(\alpha+i+1\right)^{2}}\right]}{2r_{2}\left(\frac{1}{\left(\alpha+m+1\right)\left(\alpha+i+1\right)}-\frac{c_{\alpha,\delta}^{2}}{\left(\alpha+n+1\right)\left(\alpha+j+1\right)}\right)}\right)
=  \arccos\left(\frac{\mathcal{O}\left(\alpha^{5}\right)}{2r_{2}\left(1+r_{2}^{2}\right)\alpha^{7}}\right).
\end{align*}
Since
$$\frac{\mathcal{O}\left(\alpha^{5}\right)}{2r_{2}\left(1+r_{2}^{2}\right)\alpha^{7}}\rightarrow0
\qquad\text{as}\qquad\alpha\rightarrow\infty,
$$
this means that $\theta_{\alpha}\rightarrow\frac{\pi}{2}$ when $\alpha\rightarrow\infty$, 
as claimed. In particular, this shows that for all $\frac{\pi}{2}<\theta\le\pi$,
there exists an $\alpha$ for which $F_{\alpha}(\theta)<0$.
For such $\theta$ then, the Toeplitz operator with the symbol ${a_{1}z^{m}\bar{z}^{n}+r_{2}e^{i\theta}z^{i}\bar{z}^{j}}$
is not hyponormal.
\end{proof}
The next example will demonstrate that the extra conditions we used
for necessity in Theorem \ref{thm:Same Relative Degree Binomial}
cannot be completely dropped.
\begin{example}
\label{exa:CounterExampleQuarterPlane} Let $\varphi_{\theta}(z)=z^{2}\bar{z}+\frac{1}{10}e^{i\theta}z^{3}\bar{z}^{2}.$
Here again, thinking in terms of two circles as in the proof of Theorem
\ref{thm:Same Relative Degree Binomial}, we see that in this case
the interiors of the two circles are disjoint for small $\alpha$ as shown in Figure \ref{FigNotInside}.

\begin{figure}[h]
\includegraphics[scale=0.9]{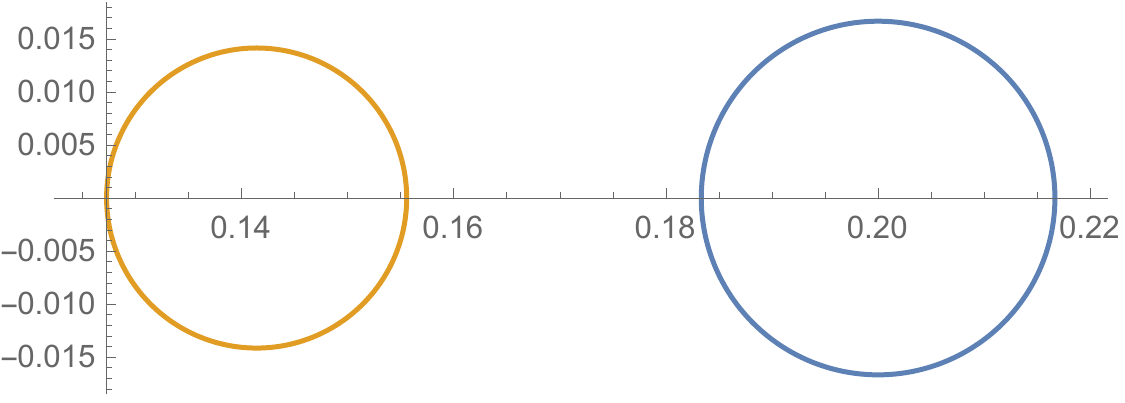}
\caption{ The situation for $\varphi_{\theta}(z)=z^{2}\bar{z}+\frac{1}{10}e^{i\theta}z^{3}\bar{z}^{2}$ with $\alpha=2$.}
\label{FigNotInside}

\end{figure}

And indeed, as $\alpha\rightarrow\infty$, and these two circles come
together, we find that 
\[
\left(\frac{1}{\alpha+3}+\frac{\cos\theta}{10(\alpha+4)}\right)^{2}+\frac{\sin^{2}\theta}{100\left(\alpha+4\right)^{2}}-c_{\alpha,\delta}^{2}\left[\left(\frac{1}{\alpha+2}+\frac{\cos\theta}{10(\alpha+3)}\right)^{2}+\frac{\sin^{2}\theta}{100\left(\alpha+3\right)^{2}}\right]>0
\]
for all $\theta\in\left[0,\pi\right]$ and all $\alpha\ge1$. Thus, the Toeplitz 
operator with symbol ${\varphi_{\theta}}$ is hyponormal for all choices of
$\theta$. 
\end{example}
 The next theorem improves slightly on the conditions of
Corollary \ref{thm:Same degree polynomial1}. 
\begin{thm}
\label{thm:Same Degree Polynomial2}Let $\varphi(z)=a_{1}z^{m_{1}}\bar{z}^{n_{1}}+\ldots+a_{k}z^{m_{k}}\bar{z}^{n_{k}}$,
with $m_{1}-n_{1}=\ldots=m_{k}-n_{k}=\delta\geq0$, and $a_{i}$ all
lying in the same quarter-plane $1\leq i\leq k$ (i.e.~$\max_{1\leq i,j\leq k}\left|\mathrm{arg}(a_{i})-\mathrm{arg}\left(a_{j}\right)\right|\leq\frac{\pi}{2}$),
then $T_{\varphi}$ is hyponormal. 
\end{thm}
\begin{proof}
The proof follows \emph{mutatis mutandis }the proof of Theorem \ref{thm:Same Relative Degree Binomial}.
Recall again that $c_{\alpha,\delta}=\sqrt{\frac{\alpha-\delta+1}{\alpha+\delta+1}}$.

We assume without loss of generality that $a_{1}$ is a positive real
number with largest modulus among the $a_{i}$'s, and with $a_{j}=r_{j}e^{i\theta_{j}}$
for $2\leq j\leq k$. The only other change is that instead of condition \eqref{eq:Two-term bound},
we have hyponormality if and only if 
\begin{align}
 &\quad \left(\frac{a_{1}}{\alpha+m_{1}+1}+\sum_{i=2}^{k}\frac{r_{i}\cos\theta_{i}}{\alpha+m_{i}+1}\right)^{2}+\left(\sum_{i=2}^{k}\frac{r_{i}\sin\theta_{i}}{\alpha+m_{i}+1}\right)^{2}\label{eq:Multiterm bound}\\
&\geq  c_{\alpha,\delta}^{2}\left(\left(\frac{a_{1}}{\alpha+n_{1}+1}+\sum_{i=2}^{k}\frac{r_{i}\cos\theta_{i}}{\alpha+n_{i}+1}\right)^{2}+\left(\sum_{i=2}^{k}\frac{r_{i}\sin\theta_{i}}{\alpha+n_{i}+1}\right)^{2}\right),\nonumber 
\end{align}
for all $\alpha\ge\delta$.

Thus, in addition to needing all $a_{i}$ in the right-half plane
to guarantee 
\[
\left(\frac{a_{1}}{\alpha+m_{1}+1}+\sum_{i=2}^{k}\frac{r_{i}\cos\theta_{i}}{\alpha+m_{i}+1}\right)^{2}\geq c_{\alpha,\delta}^{2}\left(\frac{a_{1}}{\alpha+n_{1}+1}+\sum_{i=2}^{k}\frac{r_{i}\cos\theta_{i}}{\alpha+n_{i}+1}\right)^{2},
\]
for all $\alpha\ge\delta$, we also need all $a_{i}$ in the upper
half-plane to guarantee 
\[
\left(\sum_{i=2}^{k}\frac{r_{i}\sin\theta_{i}}{\alpha+m_{i}+1}\right)^{2}\ge c_{\alpha,\delta}^{2}\left(\sum_{i=2}^{k}\frac{r_{i}\sin\theta_{i}}{\alpha+n_{i}+1}\right)^{2}.
\]
Thus, as long as $0\leq\theta_{i}\le\frac{\pi}{2}$, for all $i$,
$T_{\varphi}$ is hyponormal.

By rotation, it is sufficient to have $\max_{1\leq i,j\leq k}\left|\mathrm{arg}(a_{i})-\mathrm{arg}\left(a_{j}\right)\right|\leq\frac{\pi}{2}$. 
\end{proof}

It is not known whether or not this condition is necessary. It may be possible \emph{a priori}  to construct
$\varphi$ in such a way that condition \eqref{eq:Multiterm bound} holds, while allowing
one of the $a_{i}$ to be outside the given quarter-plane. We expect that the techniques of Example \ref{exa:CounterExampleQuarterPlane}
can be modified to yield the desired outcome.

\section{Some spectral estimates}\label{s-spectral}
We discuss some question of the spectral properties of hyponormal
$T_{\varphi}$. By Putnam's inequality, and the spectral mapping theorem,
we know that when $\varphi$ is analytic we have that
\[
\Vert[T_{\varphi}^{*},T_{\varphi}]\Vert\leq\frac{\mathrm{Area}(\varphi(\mathbb{D}))}{\pi}.
\]
However, in \cite{OlsenReguera}, it was shown that if $\varphi$ is
also univalent in $\mathbb{D}$, that 
\[
\Vert[T_{\varphi}^{*},T_{\varphi}]\Vert\leq\frac{\mathrm{Area}(\varphi(\mathbb{D}))}{2\pi},
\]
and it is a standing conjecture that the univalent condition can be
dropped. Evidence for this conjecture was given in \cite{FleemanKhavinson}
where it was showed that 
\[
\Vert[T_{z^{k}}^{*},T_{z^{k}}]\Vert=\frac{1}{2}.
\]
An examination of Theorem \ref{thm:Hyponormal monomials} lets us
show the following.
\begin{thm}
\label{PutnamLikeTheorem}
Let $\varphi(z)=z^{m}\bar{z}^{n}$ with $m>n$. Then

\[
\Vert[T_{z^{m}\bar{z}^{n}}^{*},T_{z^{m}\bar{z}^{n}}]\Vert\le\frac{1}{2}.
\]
\end{thm}
\begin{proof}
Recall that by Theorem \ref{thm:Hyponormal monomials}, we have that.
$T_{z^{m}\bar{z}^{n}}$ is hyponormal. We will write here $u(z)=\sum_{k=0}^{\infty}u_{k}\varphi_{k}(z)$,
where $\varphi_{k}(z)=\sqrt{\frac{n+1}{\pi}}z^{k}$, so that $\left\{ \varphi_{k}\right\} _{k=0}^{\infty}$
is the standard orthonormal basis of $A^{2}(\mathbb{D})$, and $\left\Vert u\right\Vert _{A^{2}(\mathbb{D})}^{2}=\sum\left|u_{k}\right|^{2}$. We obtain
\begin{align*}
\Vert[T_{z^{m}\bar{z}^{n}}^{*},T_{z^{m}\bar{z}^{n}}]\Vert & =\sup_{
\tiny{
\begin{array}{c}
u\in A^{2}(\mathbb{D}),\\
\small \Vert u\Vert=1
\end{array}
}
}
\left\langle [T_{z^{m}\bar{z}^{n}}^{*},T_{z^{m}\bar{z}^{n}}]u,u\right\rangle \\
 & =\sum_{k=0}^{m-n-1}\frac{\left(k+m-n+1\right)\left(k+1\right)}{(k+m+1)^{2}}\left|u_{k}\right|^{2}
 \\
 &\qquad+\sum_{k=m-n}^{\infty}\left(\frac{k+m-n+1}{(k+m+1)^{2}}-\frac{k+n-m+1}{(k+n+1)^{2}}\right)\left(k+1\right)\left|u_{k}\right|^{2}.
\end{align*}
The goal will be to show
\begin{equation}
\frac{\left(k+m-n+1\right)\left(k+1\right)}{(k+m+1)^{2}}\le\frac{1}{2}\label{eq:SpectralInequality1}
\end{equation}
for all $0\le k\le m-n-1$, and that
\begin{equation}
\left(\frac{k+m-n+1}{(k+m+1)^{2}}-\frac{k+n-m+1}{(k+n+1)^{2}}\right)\left(k+1\right)\le\frac{1}{2},\label{eq:SpectralInequality2}
\end{equation}
for all $k\ge m-n.$ If we accomplish these two estimates \eqref{eq:SpectralInequality1} and \eqref{eq:SpectralInequality2}, then we will have
\[
\sup_{
\tiny{
\begin{array}{c}
u\in A^{2}(\mathbb{D}),\\
\small \Vert u\Vert=1
\end{array}
}
}
\left\langle [T_{z^{m}\bar{z}^{n}}^{*},T_{z^{m}\bar{z}^{n}}]u,u\right\rangle \le\frac{1}{2}\sum_{k=0}^{\infty}\left|u_{k}\right|^{2},
\]
and the result will be proved. 

To this end, we look first at the case where $0\le k\le m-n-1$. Differentiating
$\frac{\left(k+m-n+1\right)\left(k+1\right)}{(k+m+1)^{2}}$ with respect
to $k$, we find this derivative is
\[
\frac{\left(k+m+1\right)^{2}\left(2k+2+m-n\right)-2\left(k+m+1\right)\left(k+m-n+1\right)\left(k+1\right)}{\left(k+m+1\right)^{4}}.
\]
Since $\left(k+m+1\right)^{2}\ge\left(k+m+1\right)\left(k+m-n+1\right)$,
and $\left(2k+2+m-n\right)\ge2\left(k+1\right)$, we have that this
derivative is positive for all $k\ge0$. Thus it suffices to show
that \eqref{eq:SpectralInequality1} holds when $k=m-n-1$. In this
case we have 
\[
\frac{2\left(m-n\right)^{2}}{\left(2m-n\right)^{2}}.
\]
But then since $\left(2m-n\right)^{2}=\left(m+m-n\right)^{2}=m^{2}+2m\left(m-n\right)+\left(m-n\right)^{2}$
is clearly greater than $4\left(m-n\right)^{2}$, we have that \eqref{eq:SpectralInequality1}
is bounded above by 
\[
\frac{2\left(m-n\right)^{2}}{4\left(m-n\right)^{2}}=\frac{1}{2}
\]
as desired. It remains to show that \eqref{eq:SpectralInequality2}
holds for $k\ge m-n$. To show this we note that 
\[
\left(\frac{k+m-n+1}{(k+m+1)^{2}}-\frac{k+n-m+1}{(k+n+1)^{2}}\right)\left(k+1\right)\le\frac{2\left(m-n\right)\left(k+1\right)}{\left(k+m+1\right)^{2}},
\]
since $\frac{k+n-m+1}{(k+n+1)^{2}}\ge\frac{k+n-m+1}{(k+m+1)^{2}}$.
Now, $(k+m+1)^{2}=\left(k+1\right)^{2}+2m\left(k+1\right)+m^{2}$.
It is clear that $2m\left(k+1\right)\ge2\left(m-n\right)\left(k+1\right)$.
If we can then show that 
\[
\left(k+1\right)^{2}+m^{2}\ge2\left(m-n\right)\left(k+1\right),
\]
the claim will follow. But indeed, we obtain 
\[
\left(k+1-m\right)^{2}=\left(k+1\right)^{2}-2m\left(k+1\right)+m^{2}\ge0
\]
 for all $k$, so we see
\[
\left(k+1\right)^{2}+m^{2}\ge2m\left(k+1\right)\ge2\left(m-n\right)\left(k+1\right).
\]
And it follows then that 
\[
\frac{2\left(m-n\right)\left(k+1\right)}{\left(k+m+1\right)^{2}}\le\frac{2\left(m-n\right)\left(k+1\right)}{4\left(m-n\right)\left(k+1\right)}=\frac{1}{2},
\]
and so, as claimed, the theorem is proved.
\end{proof}
The result is, perhaps, not so surprising, given the results of \cite{FleemanKhavinson}
and \cite{OlsenReguera}, but it leads us to conjecture:
\begin{center}
{\em 
If $\varphi\in L^{\infty}(\mathbb{D})$,
and if $T_{\varphi}$ is hyponormal, then 
$
\Vert[T_{\varphi}^{*},T_{\varphi}]\Vert\leq\frac{\mathrm{Area}(\varphi(\mathbb{D}))}{2\pi}.
$
}
\end{center}

\section{Final remarks}

Our studies have focused on finding sufficient conditions for the
hyponormality of Toeplitz operators having certain non-harmonic polynomials
as symbols, with our methods invariably focusing on what can only
be described as ``hard'' analysis. We would be interested in finding
more function theoretic results akin to P.~Ahern and Z.~\v{C}u\v{c}kovi\'{c}
in \cite{AhernCuckovic}, which would generate softer proofs and more
qualitative results. For example, something along the lines of the following conjecture:
\begin{center}
{\em 
If
$T_{f}$ is hyponormal and $T_{g}$ is co-hyponormal, then $T_{f+g}$
is hyponormal implies that $\left|f_{z}\right|\ge\left|g_{\bar{z}}\right|$
in $\mathbb{D}$.
}
\end{center}

So far, all the examples we have conform to this prediction, but given the subtlety of hyponormality, this evidence is certainly not overwhelming.

We would also be interested in looking at necessary conditions, along
the lines of much of the work that has been done by others studying
operators with harmonic symbols such as Z.~\v{C}u\v{c}kovi\'{c} and R.~Curto's recent work in \cite{CuckovicCurto}.

Finally, in \cite{ChuKhavinson}, Ch.~Chu and D.~Khavinson proved the following theorem for hyponormal Toeplitz operators acting on the Hardy space.

\begin{thm*}
If $\varphi=f+\overline{T_{\overline{h}}f}$ for $f,h\in H^{\infty}$,
with $\left\Vert h\right\Vert _{\infty}\le1$ and $h\left(0\right)=0$,
that is, if $T_{\varphi}$ is a hyponormal Toeplitz operator on the Hardy space
$H^{2}$, then we have that 
\[
\Vert[T_{\varphi}^{*},T_{\varphi}]\Vert\ge\left\Vert P_{+}\left(\varphi\right)-\varphi\left(0\right)\right\Vert _{2}^{2}.
\]
where $P_{+}$ is the orthogonal projection from $L^{2}(\mathbb{T})$ onto the Hardy space.
\end{thm*}
Combining this result with Putnam's inequality they arrive immediately
at the following corollary:
\begin{cor*}
If $T_{\varphi}$ is a hyponormal Toeplitz operator acting on the Hardy space
$H^{2}$, then 
\[
\mathrm{Area}\left(\sigma\left(T_{\varphi}\right)\right)\ge\pi\left\Vert P_{+}\left(\varphi\right)-\varphi\left(0\right)\right\Vert _{2}^{2}.
\]
\end{cor*}
Although a classification of hyponormal Toeplitz operators remains elusive for the Bergman space, it would be interesting to see under what conditions a similar lower bound could be obtained in the Bergman space setting. A cursory examination of the proof of Theorem \ref{PutnamLikeTheorem} combined with Putnam's inequality shows that
\[
\left\Vert P(z^{m}\bar{z}^{n})\right\Vert_{2} ^{2}=\frac{\left(m-n+1\right)\pi}{\left(m+1\right)^{2}}\le\mathrm{Area}\left(\sigma\left(T_{z^{m}\bar{z}^{n}}\right)\right).
\]


%
%
%

\begin{thebibliography}{14}
\bibitem{AhernCuckovic}P.~Ahern, Z.~\v{C}u\v{c}kovi\'{c}, \emph{A
mean value inequality with applications to Bergman space operators},
Pacific J.~Math.~173 (1996), no. 2, 295\textendash 305.

\bibitem{AxlerShapiro}S.~Axler, J.H.~Shapiro, \emph{Putnam's theorem,
Alexander's spectral area estimate, and VMO}, Math.~Ann.~271, 1985,
161\textendash 183.

\bibitem{ChuKhavinson} Ch.~Chu, D.~Khavinson, \emph{A note on the spectral area of Toeplitz operators},
Proc. Amer. Math. Soc., Vol.~144, No.~6 (2016), 2533\textendash 2537.

\bibitem{Cowen}C.~Cowen, \emph{Hyponormal and subnormal Toeplitz
operators}, Surveys of some recent results in operator theory, Vol.~I, 155\textendash 167, Pitman Res.~Notes Math.~Ser., 171, Longman
Sci.~Tech., Harlow, 1988.

\bibitem{CuckovicCurto}Z.~\v{C}u\v{c}kovi\'{c}, R.~Curto, \emph{A
New Necessary Condition for the Hyponormality of Toeplitz Operators
on the Bergman Space}, arXiv:1610.09596, 2016.

\bibitem{DurenSchuster}P.~Duren, A.~Schuster, \emph{Bergman Spaces},
Mathematical Surveys and Monographs v.~100, 2004.

\bibitem{FleemanKhavinson}M.~Fleeman, D.~Khavinson, \emph{Extremal
domains for self-commutators in the Bergman space}, Complex Anal.~Oper.~Theory 9 (2015), no.~1, 99\textendash 111.

\bibitem{Hwang}I.S.~Hwang, \emph{Hyponormal Toeplitz operators on
the Bergman space, }J.~Korean Math.~Soc.~42 (2005), no.~2, 387\textendash 403.

\bibitem{Liu-Lu}Y.~Lu, C.~Liu, \emph{Commutativity and Hyponormality
of Toeplitz operators on the weighted Bergman space, }J.~Korean Math.~Soc.~46 (2009), no.~3, 621\textendash 642.

\bibitem{MartinPutinar}M.~Martin, M.~Putinar, \emph{Lectures on Hyponormal
Operators, }Operator Theory: Advances and Applications, 39, 1989.

\bibitem{OlsenReguera}J.-F.~Olsen, M.~Reguera, \emph{On a sharp
estimate for Hankel operators and Putnam's inequality}, Rev.~Mat.~Iberoam.~32 (2016), no.~2, 495\textendash 510.

\bibitem{Rudin}W.~Rudin, \emph{Functional Analysis}, McGraw-Hill, 1973.

\bibitem{Sadraoui}H.~Sadraoui, \emph{Hyponormality of Toeplitz operators
and Composition operators}, Thesis, Purdue University, 1992. 

\bibitem{Sarason}D.~Sarason, \emph{Generalized Interpolation in $H^{\infty}$,
}Trans.~Amer.~Math.~Soc.~\textbf{127 }(1967) 179\textendash 203.
\end{thebibliography}
\end{document}